\documentclass[notitlepage]{amsproc}
\usepackage{amsfonts}
\usepackage[latin1]{inputenc}
\usepackage[all]{xypic}

\newtheorem{corollary}{\textbf{Corollary}}
\newtheorem{definition}[corollary]{\textbf{Definition}}
\newtheorem{example}[corollary]{\textbf{Example}}

\newtheorem{lemma}[corollary]{\textbf{Lemma}}

\newtheorem{refthm}[corollary]{\textbf{Theorem}}

\newcommand{\rk}{\rm rk\,}

\newcommand{\mor}[3]{$\xymatrix@1@C=15pt{#3: #1\ar[r]& #2}$}
\newcommand{\iso}[3]{$\xymatrix@1@C=15pt{#3: #1\ar[r]^-{\cong}& #2}$}

\newcommand{\psidim}{{\psi\,\rm{dim}}}
\newcommand{\phidim}{{\phi\,\rm{dim}}}

\renewcommand{\mod}{{\rm mod\,}}
\newcommand{\pd}{{\rm pd\,}}

\usepackage{latexsym,amssymb,amscd}
\usepackage{amsmath}
\usepackage[all]{xy}

\begin{document}
\thanks{The authors thank the financial supports received from Proyecto FCE-ANII 059 URUGUAY and NSERC Discovery Grants Program CANADA}
\author{ Fran\c cois Huard,\\ Marcelo Lanzilotta}
\sloppy
\bibliographystyle{plain}

\title{Self-injective right artinian rings and Igusa Todorov functions\footnote{\today}}

\begin{abstract}  We show that a right artinian ring $R$ is right self-injective if and only if $\psi(M)=0$ (or equivalently $\phi(M)=0$) for all finitely generated right $R$-modules $M$, where $\psi , \phi : \mod R \to \mathbb N$ are functions defined by Igusa and Todorov. In particular, an artin algebra $\Lambda$ is self-injective if and only if $\phi(M)=0$ for all finitely generated right $\Lambda$-modules $M$.
\end{abstract}

\maketitle

In their paper \cite{IT},  Igusa and Todorov introduce two functions $\phi$ and $\psi$ in order to show that the finitistic dimension of an artin algebra with representation dimension at most three is finite.  It turns out that these invariants also characterise self-injective right artinian rings.  We start by recalling the definitions of $\phi$ and $\psi$.  In what follows, $R$ is a right artinian ring and $\mod R$ is the category of finitely generated right $R$-modules.

Let $K$ be the free abelian group generated by all symbols $[M]$ with $M\in\mod R$ modulo the subgroup generated by:

\begin{itemize}
\item[(a)] $[A]-[B]-[C]$ if $A\cong B\oplus C$,

\item[(b)] $[P]$ if $P$ is projective.
\end{itemize}

\noindent Then $K$ is the free abelian group generated by all isomorphism classes of finitely generated indecomposable non projective $R$-modules. The syzygy functor $\Omega$ then gives rise to a group homomorphism    $\Omega: K \to K$.
For any $M\in\mod R$, let $\langle M\rangle$ denote the subgroup of $K$ generated by all the indecomposable non projective summands of $M$.   Since   the rank of $\Omega (\langle M\rangle)$ is less or equal to the rank of $\langle M\rangle$ which is finite, it follows from the well ordering principle that there exists a non-negative integer $n$ such  that the rank of $\Omega^n(\langle M\rangle)$ is equal to the rank of $\Omega^i(\langle M\rangle)$ for all $i\geq n$.  We let $\phi (M)$ denote the least such $n$.

The main properties of $\phi$ are summarized below.

\begin{lemma}\label{lem:itphi}{\cite{IT,HLM1}}
Let $R$ be a right artinian ring and $M,N\in\mod R$.
\begin{itemize}
\item[(a)]If the projective dimension of $M$, $\pd M$, is finite, then $\pd M=\phi(M)$,
\item[(b)]If $M$ is indecomposable of infinite projective dimension, then $\phi(M)=0$,
\item[(c)]$\phi(N\oplus M)\geq \phi(M)$,
\item[(d)]$\phi(M^k)=\phi(M)$ if $k \geq 1$,
\item[(e)]$\phi(M) \leq \phi(\Omega M) +1$.
\end{itemize}
\end{lemma}

The function $\psi:\mod R \to \mathbb N$ is then defined as follows.  For any $M\in\mod R$,

$\psi(M)=\phi(M)+\max\{\pd X | X \hbox{ is a summand of } \Omega^{\phi(M)}M \hbox{ and } \pd X < \infty\}. $

\bigskip

\begin{lemma}\label{lem:itpsi}{\cite{IT,HLM1}}
Let $R$ be a right artinian ring and $M,N\in\mod R$.
\begin{itemize}
\item[(a)]If the projective dimension of $M$ is finite, then $\pd M=\psi(M)$,
\item[(b)]If $M$ is indecomposable of infinite projective dimension, then $\psi(M)=0$,
\item[(c)]$\psi(N\oplus M)\geq \psi(M)$,
\item[(d)]$\psi(M^k)=\psi(M)$ if $k \geq 1$,
\item[(e)]$\psi(M) \leq \psi(\Omega M) +1$,
\item[(f)]If $0\rightarrow A \rightarrow B \rightarrow C \rightarrow 0$ is a short exact sequence in $\mod R$, and $\pd C$ is finite, then $\Psi(C)\leq \Psi(A\oplus B) +1$.
\end{itemize}
\end{lemma}

We introduce the natural concepts of $\phi$-dimension and  $\psi$-dimension for a right artinian ring $R$.

\begin{definition} For a right artinian ring $R$, $\phidim(R)=\sup\{\phi(M) | M \in\mod R\} $ and
$\psidim(R)=\sup\{\psi(M) | M \in\mod R\} .$
\end{definition}

Another invariant for $R$ is its finitistic dimension, fin.dim$(R)$,  defined as  the supremum of the projective dimensions of the finitely generated right $R$-modules of finite projective dimension (see \cite{ZH}).  It follows from Lemma \ref{lem:itphi}(a) and Lemma \ref{lem:itpsi}(a) that if $\phidim R$ or $\psidim R$ is finite, then the finitistic dimension of $R$ is also finite.

 Recall that a right artinian ring $R$ is right self-injective if the module $R_R$ is injective.  Note that every indecomposable module over a right self-injective right artinian ring is either projective or has infinite projective dimension.  However, this property does not characterize right self-injective rings.

 \begin{example}

Consider $\Lambda$ the bound quiver algebra $kQ/J^2$ where $k$ is a field, $J$ is the ideal of $KQ$ generated by the arrows and $Q$ is given by

$$ \xymatrix{1\ar[r]& 2\ar[r]& 3\ar[r] & \cdots \ar[r] & n-1\ar[r] &n \ar@(ur,dr)   }$$

\bigskip

\noindent  In this case, the projective dimension of each finitely generated right $\Lambda$-module is either zero or infinite.  Therefore the finitistic dimension of $R$ is equal to zero.  However, $\psidim (\Lambda) =\phidim (\Lambda)=\phi(S_1 \oplus S_n)=n-1$ where $S_1$ and $S_n$ denote the simple modules at the vertices $1$ and $n$ respectively.  Note that $\Lambda$ is not self-injective since the indecomposable projective at the vertex $n$ is not injective.

\end{example}

For any right artinian ring $R$, we have fin.dim$(R)\leq \phidim (R) \leq$ gl.dim$(R)$, where gl.dim$(R)$ denotes the global dimension of $R$.  As the example above shows, these inequalities can be strict.  We can now state and prove our main result.

\begin{refthm}  For a right artinian ring $R$, the following are equivalent.
\begin{itemize}
\item[(a)]$\phidim(R)=0$.
\item[(b)]$\psidim(R)=0$.
\item[(c)]$R$ is right self-injective.
\end{itemize}
\end{refthm}
\begin{proof}

Clearly, (b) implies (a).  On the other hand, if $\phidim(R)=0$, then for each $R$-module $M$, either $\pd M=0$ or $\pd M=\infty$.  Thus $\phi(M)=\psi(M)$ for all modules $M$ and hence $\psidim(R)=0$.

We will now prove that (a) and (c) are equivalent.  Assume that $\phidim(R)=0$.  We start by showing that each indecomposable projective has simple socle. Let $P$ be an indecomposable projective module and assume that $P$ has two nonisomorphic simples $S_1$ and $S_2$ in its socle.  This yields the short exact sequences:

\smallskip

$0\rightarrow S_1 \rightarrow P \rightarrow M_1 \rightarrow 0,$

\smallskip

$0\rightarrow S_2 \rightarrow P \rightarrow M_2 \rightarrow 0,$

\smallskip

$0\rightarrow S_1\oplus S_2 \rightarrow P \rightarrow M_3 \rightarrow 0,$

\smallskip

\noindent with $M_1$, $M_2$, $M_3$ nonisomorphic indecomposable (since they have simple top) modules.  But then the rank of $\langle M_1 \oplus M_2 \oplus M_3\rangle$ is 3 while the rank of $\Omega(\langle M_1 \oplus M_2 \oplus M_3\rangle)$ is 2, implying that $\phi(M_1\oplus M_2 \oplus M_3) \geq 1$, a contradiction. Assume now that $P$ has two isomorphic simples $S$ in its socle.  This yields the short exact sequences:

\smallskip

$0\rightarrow S \rightarrow P \rightarrow M_1 \rightarrow 0,$

\smallskip

$0\rightarrow S\oplus S \rightarrow P \rightarrow M_2 \rightarrow 0,$

\smallskip

\noindent with $M_1$ and $M_2$ nonisomorphic indecomposable (since they have simple top) modules. The rank of $\langle M_1 \oplus M_2\rangle$ is 2 while the rank of $\Omega(\langle M_1 \oplus M_2\rangle)$ is 1, thus $\phi(M_1\oplus M_2) \geq 1$, a contradiction.
Thus each indecomposable projective module has simple socle.

Given an indecomposable projective $P\in\mod R$, let $I$ be its injective hull.  Since $P$ has simple socle, $I$ must be indecomposable.  Note that since $R$ is right artinian, $I$ is not necessarily finitely generated.  We will show that $P$ is injective.  If not, we have the following commutative diagram

$$\xymatrix{
 0 \ar[r] & P \ar[d] \ar[r]& U \ar[d]\ar[r]& S\ar[d] \ar[r]  & 0\\
 0 \ar[r] &  P\ar[r] & I \ar[r] & C \ar[r] & 0      }$$

 \bigskip

\noindent where $C\neq 0$, $S$ is a simple $R$-module lying in the socle of $C$, and the upper sequence is obtained by lifting the monomorphism $S\to C$.    Note that since $P$ and $S$ are finitely generated, so is $U$.  Moreover,  the map $U\to I$ is a monomorphism  and hence $U$ is indecomposable since it has simple socle. This implies that $S$ is not projective, and that a fortiori neither is $U$ since otherwise we would have $\pd S=1=\phi(S)$, a contradiction.  Let $P(S)$ be the projective cover of $S$. We have the following commutative diagram

$$\xymatrix{ {} & {} & 0 \ar[d] & 0\ar[d] & {}\\
{} & 0 \ar[d] & \Omega(U)\oplus P'' \ar[d] \ar[r]_\cong & \Omega(S) \ar[d]  & {}\\
 0 \ar[r] & P \ar[d] \ar[r]& P' \ar[d]\ar[r]& P(S)\ar[d] \ar[r]  & 0\\
 0 \ar[r] &  P\ar[d] \ar[r] & U \ar[d]\ar[r] & S \ar[d]\ar[r] & 0\\
 {} & 0 & 0 & 0 & {}       }$$

\bigskip

\noindent where the isomorphism follows from the snake lemma.  In $K$, we have
$[\Omega(S)]=[\Omega(U) \oplus P'']=[\Omega(U)]$.  Therefore, since $U$ and $S$ are indecomposable, we have $\Omega(\langle U\rangle)= \Omega(\langle S\rangle)$, so that $\Omega(\langle U\oplus S\rangle)= \Omega(\langle U\rangle)+\Omega(\langle S\rangle)=\Omega(\langle S\rangle)$.  Now $S$ is not a summand of $U$ and neither of them are projective hence we have $\rk \langle U\oplus S\rangle > \rk \langle S \rangle = \rk  \Omega(\langle S \rangle) =\rk \Omega(\langle U\oplus S\rangle)$ where the first equality follows from the hypothesis that $\phidim(R)=0$. Thus $\phi(U\oplus S)>0$, a contradiction.  Hence $P$ is injective.  Since this holds for every indecomposable $R$-projective $P$, $R$ is right self-injective.

Assume now that $R$ is right self-injective and  let $M\in\mod R$.  Since we wish to compute $\phi(M)$, we can assume that all summands of $M$ are non projective.  Hence for each  summand $M'$ of $M$, $\pd M'=\infty$.  Let $M_1, M_2$ be direct summands of $M$.  We claim that $\Omega^n M_1 \cong \Omega^n M_2 \iff M_1 \cong M_2$.  Indeed, if we consider the n-th sysygy of $M_1$ and $M_2$,
we have

\smallskip

$0\rightarrow \Omega^nM_1\rightarrow P_{n-1} \rightarrow P_{n-2} \rightarrow \ldots\rightarrow P_0\rightarrow M_1\rightarrow 0,$

\smallskip

$0\rightarrow \Omega^nM_2\rightarrow Q_{n-1} \rightarrow Q_{n-2} \rightarrow \ldots\rightarrow Q_0\rightarrow M_2\rightarrow 0,$

\smallskip

\noindent where the $P_i$'s and the $Q_i$'s are projective and hence injective $R$-modules.  The given resolutions are then injective resolutions. Consequently if $\Omega^n M_1 \cong \Omega^n M_2$, then $M_1\cong \Omega^{-n}(\Omega^n M_1) \cong \Omega^{-n}(\Omega^n M_2)\cong M_2$.  But then $\rk \langle M \rangle = \rk \Omega^i(\langle M \rangle)$ for each $i\geq 0$, implying that $\phi(M)=0$.  Since this holds for all $M\in\mod R$, we showed that $\phidim(\Lambda)=0$.
\end{proof}

Recall that a ring $R$ is self-injective if $_RR$ and $R_R$ are injective left and right $R$-modules respectively. Using the duality of an artin algebra and the fact that the  number of isomorphism classes of indecomposable projective and indecomposable injective modules are the same, the following corollary is immediate.

\begin{corollary} Let $\Lambda$ be an artin algebra.  Then $\Lambda$ is self-injective if and only if $\phidim(\Lambda)=0$.
\end{corollary}

\vskip3mm \noindent Fran\c cois Huard:\\ Department of mathematics,
Bishop's University,\\ Sherbrooke, Qu\'ebec, CANADA,  J1M1Z7.\\
{\tt fhuard@ubishops.ca}

\vskip3mm \noindent Marcelo Lanzilotta:\\ Instituto de Matem\'atica y Estad\'{i}stica Rafael Laguardia,\\
J. Herrera y Reissig 565, Facultad de Ingenier\'{i}a, Universidad de la Rep\'ublica. CP 11300, Montevideo, URUGUAY.\\
{\tt marclan@fing.edu.uy}


\begin{thebibliography}{20}


\bibitem{HLM1} F. Huard, M. Lanzilotta, O. Mendoza. An approach to the Finitistic Dimension Conjecture. {\it J. of Algebra} 319, 3918-3934, (2008).


\bibitem{IT} K. Igusa, G. Todorov. On the finitistic global dimension conjecture for artin algebras. {\it Representation of algebras and related topics,} 201-204. Field Inst. Commun., 45. Amer. Math. Soc., Providence, RI, (2005).

\bibitem{ZH} B. Zimmerman-Huisgen. The finitistic dimension  conjecture- a tale of 3.5 decades. in: Abelian groups and modules (Padova, 1994) 501-517. Math. Appl. 343, Kluwer Acad. Publ. Dordrecht, (1995).
\end{thebibliography}
\end{document}